\pdfoutput=1
\RequirePackage{ifpdf}
\ifpdf 
\documentclass[pdftex]{sigma}
\else
\documentclass{sigma}
\fi

\usepackage{diagbox}

\numberwithin{equation}{section}

\newtheorem{Theorem}{Theorem}[section]
\newtheorem{Corollary}[Theorem]{Corollary}
\newtheorem{Lemma}[Theorem]{Lemma}
\newtheorem{Proposition}[Theorem]{Proposition}
\newtheorem{Conjecture}[Theorem]{Conjecture}
 { \theoremstyle{definition}
\newtheorem{Definition}[Theorem]{Definition}
\newtheorem{Remark}[Theorem]{Remark} }

\begin{document}
\allowdisplaybreaks

\newcommand{\arXivNumber}{1907.13417}

\renewcommand{\PaperNumber}{107}

\FirstPageHeading

\ShortArticleName{Quasi-Invariants in Characteristic $p$ and Twisted Quasi-Invariants}

\ArticleName{Quasi-Invariants in Characteristic $\boldsymbol{p}$\\ and Twisted Quasi-Invariants}

\Author{Michael REN~$^\dag$ and Xiaomeng XU~$^\ddag$}

\AuthorNameForHeading{M.~Ren and X.~Xu}

\Address{$^\dag$~Department of Mathematics, Massachusetts Institute of Technology,\\
\hphantom{$^\dag$}~Cambridge, MA 02139, USA}
\EmailD{\href{mailto:mren36@mit.edu}{mren36@mit.edu}}

\Address{$^\ddag$~School of Mathematical Sciences and Beijing International Center for Mathematical Research,\\
\hphantom{$^\ddag$}~Peking University, Beijing 100871, China}
\EmailD{\href{mailto:xxu@bicmr.pku.edu.cn}{xxu@bicmr.pku.edu.cn}}

\ArticleDates{Received July 10, 2020, in final form October 17, 2020; Published online October 27, 2020}

\Abstract{The spaces of quasi-invariant polynomials were introduced by Chalykh and Veselov [\textit{Comm. Math. Phys.} \textbf{126} (1990), 597--611]. Their Hilbert series over fields of characteristic~0 were computed by Feigin and Veselov [\textit{Int. Math. Res. Not.} \textbf{2002} (2002), 521--545]. In this paper, we show some partial results and make two conjectures on the Hilbert series of these spaces over fields of positive characteristic. On the other hand, Braverman, Etingof and Finkelberg~[arXiv:1611.10216] introduced the spaces of quasi-invariant polynomials twisted by a monomial. We extend some of their results to the spaces twisted by a~smooth function.}


\Classification{81R12; 20C08}

\section{Introduction}

A polynomial in variables $x_1,\dots,x_n$ is \emph{symmetric} if permuting the variables does not change it. Another way to view symmetric polynomials is as invariant polynomials under the action of the symmetric group. A natural generalization of symmetric polynomials then arises: if $s_{i,j}$ is the operator on polynomials that swaps the variables~$x_i$ and~$x_j$, then we may consider polynomials~$F$ such that $F-s_{i,j}(F)$ vanishes to some order at $x_i=x_j$. Notably, if~$F$ is symmetric in~$x_i$ and $x_j$, then $F-s_{i,j}(F)$ vanishes to infinite order. These polynomials may be viewed as quasi-invariant polynomials of the symmetric group, and have been introduced by Chalykh and Veselov \cite{CV} in the study of quantum Calogero--Moser systems.

\begin{Definition}Let $k$ be a field, $n$ be a positive integer, and $m$ be a nonnegative integer. We say that a polynomial $F\in k[x_1,\dots,x_n]$ is $m$-quasi-invariant if \[(x_i-x_j)^{2m+1}\,|\, F(x_1,\dots,x_i,\dots,x_j,\dots,x_n)-F(x_1,\dots,x_j,\dots,x_i,\dots,x_n)\] for all $1\le i,j\le n$. Denote by $Q_m(n)$ the set of all $m$-quasi-invariant polynomials over $k$ in $n$ variables.
\end{Definition}

Here, we use the odd exponent $2m+1$ because if the right-hand side is divisible by $(x_i-x_j)^{2m}$, then it is also divisible by $(x_i-x_j)^{2m+1}$. This follows by the anti-symmetry of the right-hand side in $x_i$ and $x_j$. Note that $Q_m(n)$ is a module over the ring of symmetric polynomials over $k$ in $n$ variables. Also, as $Q_m(n)$ is a space of polynomials, it has a grading by degree. Thus, we may define a Hilbert series and a Hilbert polynomial to encapsulate the structure of $Q_m(n)$.

\looseness=-1 The motivation for studying quasiivariant polynomials arises from their relation with integrable systems. In 1971, Calogero first solved the problem in mathematical physics of determining the energy spectrum of a one-dimensional system of quantum-mechanical particles with inversely quadratic potentials~\cite{calogero1971solution}. Moser later on connected the classical variant of his problem with integrable Hamiltonian systems and showed that the classical analogue is indeed integrable~\cite{moser1975three}. These so-called Calogero--Moser systems have been of great interest to mathematicians as they connect many different fields including algebraic geometry, representation theory, deformation theory, homological algebra, and Poisson geometry. See, e.g.,~\cite{etingof2006lectures} and the references therein.

Quasi-invariant polynomials are deeply related with solutions of quantum Calogero--Moser systems as well as representations of Cherednik algebras \cite{feigin2002quasi}. As such, the structure of $Q_m(n)$, in particular freeness as a module, and its corresponding Hilbert series and polynomials have been extensively investigated by mathematicians. Introduced by Feigin and Veselov in 2001, their Hilbert series and lowest degree non-symmetric elements have subsequently been computed by Felder and Veselov \cite{felder2003action}. In 2010, Berest and Chalykh generalized the idea to quasi-invariant polynomials over an arbitrary complex reflection group \cite{berest2011quasi}. Recently in 2016, Braverman, Etingof, and Finkelberg proved freeness results and computed the Hilbert series of a generalization of~$Q_m(n)$ twisted by monomial factors \cite{braverman2016cyclotomic}. Our goal is to extend the investigation of~$Q_m(n)$ and its various generalizations.

In Section~\ref{sec:p}, we investigate quasi-invariant polynomials over finite fields. In particular, we provide sufficient conditions for which the Hilbert series over characteristic $p$ is greater than over characteristic~0. We conjecture that our sufficient conditions are also necessary. We also make conjectures about the properties of the Hilbert series over finite fields.

In Section~\ref{sec:twist}, we investigate a generalization of the twisted quasi-invariants. In~\cite{braverman2016cyclotomic}, Braverman, Etingof and Finkelberg introduced the space of quasi-invariants twisted by monomial factors, again a module over the ring of symmetric polynomials. They proved freeness results and computed the corresponding Hilbert series.
We generalize their work to the space of quasi-invariants twisted by arbitrary smooth functions and determine the Hilbert series in certain cases when there are two variables.

In Section~\ref{sec:fut}, we discuss future directions for our research, in particular considering spaces of polynomial differential operators and $q$-deformations.

\section[Quasi-invariant polynomials over fields of nonzero characteristic]{Quasi-invariant polynomials over fields\\ of nonzero characteristic}\label{sec:p}

Much of the previous research on quasi-invariant polynomials has been done over fields of characteristic zero. The general approach is to use representations of spherical rational Cherednik algebras \cite{braverman2016cyclotomic}. In the case of fields of positive characteristic, we take a different approach.

Let $k$ be $\mathbb F_p$, and $Q_m(n)$ the set of all $m$-quasi-invariant polynomials over $k$ in $n$ variables. To begin, we define the Hilbert series of $Q_m(n)$.

\begin{Definition}
Let the Hilbert series of $Q_m(n)$ be \[H_m(t)=\sum_{d\ge0}t^d\cdot\dim Q_{m,d}(n),\] where $Q_{m,d}(n)$ is the $k$ vector subspace of $Q_m(n)$ consisting of polynomials with degree $d$.
\end{Definition}

By the Hilbert basis theorem, $Q_m(n)$ is a finitely generated module over the ring of symmetric polynomials. Thus, we may write \[H_m(t)=\frac{G_m(t)}{\prod\limits_{i=1}^n(1-t^i)},\] where $G_m(t)$ is the Hilbert polynomial associated with $H_m(t)$ and the terms in the denominator correspond to the elementary symmetric polynomials that generate the ring of symmetric polynomials in $n$ variables.

We are mainly concerned with the difference between the Hilbert series of $Q_m(n)$ over characteristic $p$ and characteristic 0. The following proposition states that the Hilbert series of $Q_m(n)$ is at least as large in the former case as in the latter case.

\begin{Proposition}\label{prop:ge}
$\dim Q_{m,d}(n)$ over $\mathbb F_p$ is at least as large as over $\mathbb C$ for each choice of $m$, $n$, and $d$.
\end{Proposition}

\begin{proof}Suppose that $F=\sum_{i_1+\cdots+i_n=d}a_{i_1,\dots,i_n}x_1^{i_1}\cdots x_n^{i_n}$ is in $Q_{m,d}(n)$. Then, either $d<2m+1$, in which case we must have $F$ symmetric or $(x_i-x_j)^{2m+1}$ would divide a nonzero polynomial with degree $d$ for some choice of $i$ and $j$, a contradiction. This means that the dimensions are equal over either characteristic. Otherwise, we have that \[F-s_{i,j}F=(x_i-x_j)^{2m+1}\left(\sum_{j_1+\dots+j_n=d-(2m+1)}b_{i,j,j_1,\dots,j_n}x_1^{j_1}\cdots x_n^{j_n}\right)\] for each pair $i$, $j$. These yield a system of linear equations in the undetermined coefficients of~$F$ and $\frac{F-s_{i,j}F}{(x_i-x_j)^{2m+1}}$, which is with integral coefficients we are considering. It then follows from considering the null-space that the dimension of the solution space over a field of characteristic~$p$ is at least the dimension over a field of characteristic~0.
\end{proof}

However, for each $m$, there are only finitely many primes for which the Hilbert series of~$Q_m(n)$ is strictly greater over ${\mathbb F}_p$ than over $\mathbb C$.

\begin{Proposition}
For any fixed $m$ and $n$, there are only finitely many primes $p$ for which the Hilbert series of $Q_m(n)$ is greater over $\mathbb F_p$ than over $\mathbb C$.
\end{Proposition}

\begin{proof}
Let $P=\mathbb Z[x_1,\dots,x_n]$, $Q=\bigoplus_{1\le i<j\le n}P/(x_i-x_j)^{2m+1}P$, and $h$ be the linear map from~$P$ to~$Q$ defined as
\[ h(F)=\bigoplus_{1\le i<j\le n}(1-s_{i,j})F.\]  Note that $\operatorname{Ker}(h)$ coincides with $Q_m(n)$ by definition. Set $M=\operatorname{Coker}(h)$ as the cokernel of~$h$ in~$Q$ and note that if $Q_m(n)$ over $\mathbb F_p$ has a higher dimension than $Q_m(n)$ over $\mathbb C$ for some degree of the polynomials, then $M$ must have $p$-torsion. To prove that there are only finitely many such primes~$p$, we use the following generic freeness lemma, see, e.g., \cite[Theorem~14.4]{Eisenbud}.

\begin{Lemma} For a Noetherian integral domain $A$, a~fini\-te\-ly generated $A$-algebra $B$, and a~fini\-te\-ly generated $B$-module~$M$, there exists a nonzero element~$r$ of~$A$ such that the localization~$M_r$ is a~free $A_r$ module.
\end{Lemma}

We apply this in the case where $A=\mathbb Z$, $B=\mathbb Z[x_1,\dots,x_n]^{S_n}$ and $M=\operatorname{Coker}(h)$. It is easy to see that these satisfy the conditions for $A$, $B$, and $M$ in the lemma. Thus there exists an integer $r\in\mathbb Z\setminus\{0\}$ such that $M_r$ is free over $\mathbb Z[1/r]$. As~$M$ has no $p$-torsion for any $p\nmid r$, $M$~has no $p$-torsion for all but finitely many primes~$p$ so the Hilbert series over~$\mathbb F_p$ is the same as over~$\mathbb C$.
\end{proof}

We now determine the primes for which $Q_m(n)$ is greater. First, we examine the case when $n=2$.

\begin{Proposition}When $n=2$, the Hilbert series for $Q_m(2)$ over characteristic~$p$ coincides with that of characteristic~$0$. It is $\frac{1+t^{2m+1}}{(1-t)(1-t^2)}$ over all fields.
\end{Proposition}

\begin{proof}
We claim that the dimension of $Q_{m,d}(2)$ over $\mathbb C$ is equal to the dimension of~$Q_{m,d}(2)$ over~$\mathbb F_p$. By Proposition~\ref{prop:ge}, it suffices to show that for each $m$ and $d$, the dimension of~$Q_{m,d}(2)$ over $\mathbb C$ is at least the dimension of~$Q_{m,d}(2)$ over~$\mathbb F_p$. Consider a basis $f_1,\dots,f_k\in\mathbb F_p[x,y]$ of~$Q_{m,d}(2)$ over~$\mathbb F_p$. We will show the existence of $F_1,\dots,F_k\in\mathbb Z[x,y]$ of $Q_{m,d}(2)$ such that $F_i\equiv f_i\pmod p$ for all~$i$. This means that $F_1,\dots,F_k$ are linearly independent, as otherwise there exist relatively prime integers $n_1,\dots,n_k$ with $n_1F_1+\cdots+ n_kF_k=0$. Taking the equation modulo $p$ yields $n_1f_1+\cdots+ n_kf_k\equiv0\pmod p$, a contradiction with $f_1,\dots,f_k$ forming a basis of $Q_{m,d}(2)$ as not all of $n_1,\dots,n_k$ are divisible by~$p$.

To show the existence of such $F_1,\dots,F_k$, let $f=f_i$ for a fixed $i$ and suppose that $f(x,y)-f(y,x)=(x-y)^{2m+1}g(x,y)$ for some symmetric $g(x,y)\in\mathbb F_p[x,y]$. Let us take a symmetric $G(x,y)\in \mathbb Z[x,y]$ such that $G\equiv g\pmod{p}$. Let $f(x,y)=\sum_{i=0}^da_ix^iy^{d-i}$ and suppose that $G(x,y)(x-y)^{2m+1}=\sum_{i=1}^dB_ix^iy^{d-i}$ with $a_i\in\mathbb F_p$ and $B_i\in \mathbb Z$. We have that $a_i-a_{d-i}\equiv B_i\pmod{p}$. Note that $G(x,y)$ is symmetric, so $G(x,y)(x-y)^{2m+1}$ is anti-symmetric, which implies that $B_i+B_{d-i}=0$ for all $i$. Now, define $F(x,y)=\sum_{i=1}^dA_ix^iy^{d-i}$, where $A_i\equiv a_i\pmod{p}$ for \mbox{$i\le\frac d2$} and $A_i=A_{d-i}+B_i$ for \mbox{$i>\frac d2$}. Note that for \mbox{$i>\frac d2$}, we have that $A_i\equiv A_{d-i}+B_i\equiv a_i\pmod{p}$, so this $F$ satisfies $F\equiv f\pmod p$. It remains to check the quasi-invariance condition. However, note that
\begin{gather*}
F(x,y)-F(y,x)=\sum_{i=1}^d(A_i-A_{d-i})x^iy^{d-i}=\sum_{i=1}^dB_ix^iy^{d-i}=G(x,y)(x-y)^{2m+1}
\end{gather*} by definition, so we are done.

Hence, the dimension, and thus the series, is independent of $p$. It is known from~\cite{braverman2016cyclotomic} that the series is $\frac{1+t^{2m+1}}{(1-t)(1-t^2)}$, as desired.
\end{proof}

When $n>2$, the series differs greatly for many primes. In this case, we have found a sufficient condition for when the Hilbert series over characteristic $p$ is greater.

\begin{Theorem}\label{thm:suf}
Let $m\ge0$ and $n\ge3$ be integers. Let $p$ be a prime such that there exist integers $a\ge 0$ and $k\ge0$ with \[\frac{mn(n-2)+\binom n2}{n(n-2)k+\binom n2-1}\le p^a\le\frac{mn}{nk+1}.\] Then the Hilbert series of $Q_m(n)$ with $n$ variables over $\mathbb F_p$ is different from the Hilbert series over $\mathbb C$.
\end{Theorem}

\begin{proof}
The following formula, due to \cite{felder2003action}, gives the Hilbert polynomial for $Q_m(n)$ over $\mathbb C$: \[n!t^{m\binom n2}\sum_{\text{Young diagrams}}\prod_{i=1}^nt^{m(\ell_i-a_i)+\ell_i}\frac{1-t^i}{h_i\big(1-t^{h_i}\big)}.\]
Here, the sum is over Young diagrams with $n$ boxes, $a_i$ denotes the number of boxes to the right of the $i$th box, $\ell_i$ denotes the number of boxes below the $i$th box, and $h_i=a_i+\ell_i+1$. It is not hard to see that the formula gives that the Hilbert polynomial is of the form $1+(n-1)t^{mn+1}+\cdots$, where the exponents are sorted in ascending order. This is as the two terms of smallest degree are contributed by the Young diagrams corresponding to the partitions $(n)$ and $(n-1,1)$. This implies that all polynomials in $Q_m(n)$ with degree at most $mn$ are symmetric, and~$Q_m(n)$ as a~module over symmetric polynomials has a generator of degree $mn+1$. For any $m$, denote this generator in $Q_m(n)$ by~$P_m$. In the following construction, we will use the generator $P_k$ of $Q_k(n)$ for certain $k<m$.

To show that the Hilbert series is different over $\mathbb F_p$, we consider the following non-symmetric polynomial: \[F=P_k^{p^a}\prod_{1\le i<j\le n}(x_i-x_j)^{2b}.\] Here $b=\frac{2m+1-p^a(2k+1)}2$, and $a$, $k$ are integers such that the above inequalities are satisfied. So
\begin{gather*}
\deg F=p^a(nk+1)+2b\binom n2=p^a(nk+1)+\binom n2(2m+1-p^a(2k+1))\\
\hphantom{\deg F}{} =\binom n2(2m+1)+p^a\left(1-\binom n2-n(n-2)k\right)\\
\hphantom{\deg F}{} \le\binom n2(2m+1)-\left(mn(n-2)+\binom n2\right) =mn<\deg P_m.
\end{gather*} Hence, if we show that $F\in Q_m(n)$, then as $\deg F<\deg P_m$ we obtain a different Hilbert series over $\mathbb F_p$, in particular in the coefficient of $t^{\deg F}$. To do that, note that $b$ is an integer when~$p$ is odd and a half-integer when $p=2$. Either way, $\prod(x_i-x_j)^{2b}$ is a~symmetric polynomial in~$\mathbb F_p$, so we have that $(1-s_{i,j})\big(P_k^{p^a}\prod(x_i-x_j)^{2b}\big)=((1-s_{i,j})P_k)^{p^a}\prod(x_i-x_j)^{2b}$ by the fact that $(u+v)^{p^a}=u^{p^a}+v^{p^a}$ in $\mathbb F_p$. Hence, as $(x_i-x_j)^{2k+1}$ divides $(1-s_{i,j})P_k$ by assumption, we have that $(x_i-x_j)^{p^a(2k+1)+2b}=(x_i-x_j)^{2m+1}$ divides $(1-s_{i,j})F$. Hence, $F$ is in~$Q_m(n)$, so this produces a generator of $Q_m(n)$ of lower degree in $\mathbb F_p$ and thus a different Hilbert series over~$\mathbb F_p$, as desired.
\end{proof}

\begin{Remark}Let us write the inequalities in Theorem \ref{thm:suf} in the form
\[ k+\frac{1}{n}\le m/p^a\le k+\frac{n+1}{2n}-\frac{n-1}{2(n-2)p^a}.\]
In this way, we can rewrite it as
\[ \frac{1}{n}\le \left\{\frac{m}{p^a}\right\}\le \frac{n+1}{2n}-\frac{n-1}{2(n-2)p^a},
\]
where $\{\cdots\}$ denotes fractional part, which eliminates $k$.
Also from this form it is clear that~$a$ cannot be zero, i.e., $a\ge 1$.
\end{Remark}
\begin{Remark}
Let $k=0$ in the inequality in Theorem~\ref{thm:suf}. Then, we see that all primes $p$ with a~power $p^a$ between roughly $2m$ and $mn$ satisfy the inequality. These primes $p$ satisfy the property that $Q_m(n)$ over $\mathbb F_p$ has a different Hilbert series than over $\mathbb C$.
\end{Remark}

\begin{Conjecture}\label{conj:mn}
The sufficient condition we have given in Theorem~{\rm \ref{thm:suf}} is also necessary. That is if the Hilbert series of $Q_m(n)$ in $\mathbb F_p$ is different from the Hilbert series in $\mathbb C$, then there exist integers $a\ge 0$ and $k\ge0$ such that \[\frac{mn(n-2)+\binom n2}{n(n-2)k+\binom n2-1}\le p^a\le\frac{mn}{nk+1}.\] In particular, if $p>mn$, then the Hilbert series over $\mathbb F_p$ is the same as over $\mathbb C$.
\end{Conjecture}

This is supported by computer calculations, especially in the case of $n=3,4$. They suggest that the Hilbert series takes a form depending on the smallest non-symmetric element of $Q_m(n)$ which is described by the proof of Theorem~\ref{thm:suf} and hence satisfies the conjecture. The following table summarizes the results of our computer program verification for $n=3$, $m\le15$ and $p\le50$. Each box in which the series is greater over $\mathbb F_p$ than over $\mathbb C$ is labeled with its integers $a$, $k$ that make the inequality hold.

\begin{table}[!h]\centering\small \begin{tabular}{c|c|c|c|c|c|c|c|c|c|c|c|c|c|c|c|}
 \diagbox{$m$}{$p$}&2&3&5&7&11&13&17&19&23&29&31&37&41&43&47\\
 \hline
 0&&&&&&&&&&&&&&&\\\hline
 1&&$1,0$&&&&&&&&&&&&&\\\hline
 2&&&$1,0$&&&&&&&&&&&&\\\hline
 3&$3,0$&$2,0$&&$1,0$&&&&&&&&&&&\\\hline
 4&$3,0$&$2,0$&&&$1,0$&&&&&&&&&&\\\hline
 5&&$2,0$&&&$1,0$&$1,0$&&&&&&&&&\\\hline
 6&$4,0$&&&&$1,0$&$1,0$&$1,0$&&&&&&&&\\\hline
 7&$4,0$&$1,2$&$1,1$&&&$1,0$&$1,0$&$1,0$&&&&&&&\\\hline
 8&$4,0$&&&&&&$1,0$&$1,0$&$1,0$&&&&&&\\\hline
 9&$4,0$&$3,0$&$2,0$&&&&$1,0$&$1,0$&$1,0$&&&&&&\\\hline
 10&&$3,0$&$2,0$&$1,1$&&&$1,0$&$1,0$&$1,0$&$1,0$&&&&&\\\hline
 11&$5,0$&$3,0$&$2,0$&&&&&$1,0$&$1,0$&$1,0$&$1,0$&&&&\\\hline
 12&$5,0$&$3,0$&$2,0$&&&&&&$1,0$&$1,0$&$1,0$&&&&\\\hline
 13&$5,0$&$3,0$&$2,0$&&&&&&$1,0$&$1,0$&$1,0$&$1,0$&&&\\\hline
 14&$5,0$&$3,0$&$2,0$&&&&&&$1,0$&$1,0$&$1,0$&$1,0$&$1,0$&&\\\hline
 15&$5,0$&$3,0$&$2,0$&&$1,1$&&&&&$1,0$&$1,0$&$1,0$&$1,0$&$1,0$&\\
 \hline\end{tabular}
\end{table}

Through our programs, we have found that when $n=3$, the Hilbert series takes the form \[\frac{1+2t^d+2t^{6m+3-d}+t^{6m+3}}{(1-t)(1-t^2)(1-t^3)}\] for small $p$, where $d$ is the degree of the smallest non-symmetric generator of $Q_m$ in $\mathbb F_p$. In particular, this smallest non-symmetric polynomial in $Q_m$ is of the form $P_k^{p^a}\prod(x_i-x_j)^{2b}$ where the $P_k$ are as described in Theorem~\ref{thm:suf}. Furthermore, we conjecture that $Q_m$ is a free module over the ring of symmetric polynomials for $n=3$ over any field.

In \cite{felder2003action}, the authors prove some properties of Hilbert series and polynomials over $\mathbb C$, specifically their maximal term and symmetry. We believe that similar results still hold over $\mathbb F_p$, and this is supported by our computer calculations for $n=3,4$.

\begin{Conjecture}\label{conj:mainc}
The largest degree term in the Hilbert polynomial is always $t^{\binom n2(2m+1)}$. Furthermore, when $p$ is an odd prime $Q_m$ is a free module over the ring of the symmetric polynomials of rank~$n!$, and the Hilbert polynomial is palindromic.
\end{Conjecture}

\begin{Remark}
The condition that $p$ is odd appears to be necessary. Indeed, a computer calculation shows that when $n=4$, $m=1$ and $p=2$, the Hilbert series is \begin{gather*}
1+t+2t^2+3t^3+8t^4+9t^5+15t^6+23t^7+38t^8+50t^9+71t^{10}+\cdots\\
\qquad {} = \frac{1+3t^4+3t^7+5t^8+3t^9-t^{10}+\cdots}{(1-t)(1-t^2)(1-t^3)(1-t^4)},
\end{gather*} and the negative coefficient implies that the module cannot be free. In particular, computing the polynomial up to~$t^{18}$ also demonstrates that it is not symmetric.
\end{Remark}

\section{Twisted quasi-invariants}\label{sec:twist}

\subsection{A generalization of quasi-invariants}
In \cite{braverman2016cyclotomic}, Braverman, Etingof and Finkelberg introduced quasi-invariants twisted by a monomial $x_1^{a_1}\cdots x_n^{a_n}$, where $a_1\dots,a_n\in\mathbb C$. We further generalize this by allowing the twist to be a~product of general functions. To be more precise, let $m$ be a nonnegative integer. Fix one-variable meromorphic functions $f_1,f_2,\dots,f_n$, and denote by $D\subset \mathbb{C}^n$ the domain where the product $f_1(x_1)f_2(x_2)\cdots f_n(x_n)$ and its inverse are smooth.
\begin{Definition}
We define $Q_m(f_1,\dots,f_n)$ to be the space of polynomials $F\in\mathbb C[x_1,\dots,x_n]$ for which \[\frac{(1-s_{i,j})\left(f_1(x_1)\cdots f_n(x_n)F(x_1,\dots,x_n)\right)}{(x_i-x_j)^{2m+1}}\] is smooth on $D$ for all $1\le i<j\le n$.
\end{Definition}
In the following, for simplicity when we say smooth functions we always mean smooth on $D$.
\begin{Remark}
In \cite{braverman2016cyclotomic}, the authors studied the case $f_i(x)=x^{a_i}$ for $a_i\in\mathbb C$, and denoted $Q_m(f_1,\dots,f_n)$ by $Q_m(a_1,\dots,a_n)$. In cases of unambiguous use, we will shorten this to~$Q_m(n)$.
\end{Remark}

Similar to \cite{braverman2016cyclotomic}, we believe that $Q_m(f_1,\dots,f_n)$ is in general a free module.

\begin{Conjecture}\label{conj:big}
For generic $f_1,\dots,f_n$, in particular when $\frac{f_i}{f_j}$ is not a monomial in $x_1,\dots,x_n$, $Q_m(f_1,\dots,f_n)$ is a free module over the ring of symmetric polynomials in $x_1,\dots,x_n$.
\end{Conjecture}

\subsection{Rationality of the logarithmic derivative}\label{subsec:dlog}
Note that for any $f_1,\dots,f_n$ and $i$, $j$, any $F\in(x_i-x_j)^{2m}\mathbb{C}[x_1,\dots,x_n]$ has the property that $\frac{(1-s_{i,j})(f_1(x_1)\cdots f_n(x_n)F)}{(x_i-x_j)^{2m+1}}$ is smooth. This is a trivial case, thus we want to find out which choices of the $f_i,f_j$ yield polynomials in $Q_m(n)$ that are not in $(x_i-x_j)^{2m}\mathbb{C}[x_1,\dots,x_n]$.

\begin{Lemma}\label{prop:squ}
If $F$ is divisible by $x_i-x_j$ and $\frac{(1-s_{i,j})(f_1(x_1)\cdots f_n(x_n)F)}{(x_i-x_j)^{2(k+1)+1}}$ is smooth then $(x_i-x_j)^2\,|\, F$ and $\frac{(1-s_{i,j})\big(f_1(x_1)\cdots f_n(x_n)\frac F{(x_i-x_j)^2}\big)}{(x_i-x_j)^{2k+1}}$ is smooth.
\end{Lemma}

\begin{proof}
Let $F(x_i,x_j)=(x_i-x_j)G(x_i,x_j)$ for some polynomial $G$. Here, we write $F(x_i,x_j)$ for $F$ and $G(x_i,x_j)$ for $G$ to ease the notation, as we will only consider it as a function in the $i$th and $j$th coordinates. Substituting, the condition becomes \[f_i(x_i)f_j(x_j)G(x_i,x_j)+f_j(x_i)f_i(x_j)G(x_j,x_i)=(x_i-x_j)^{2k+2}g(x_i,x_j)\]
(here and for the rest of this section, $g$ (with a possible subscript, e.g., $g_a$, $g_b$, $g_c$) denotes a~function smooth on $D$) and setting $x_i=x_j=x$ gives \[f_i(x)f_j(x)G(x,x)=0,\] so $G(x,x)=0$. This implies that $(x_i-x_j)\,|\, G$, so $(x_i-x_j)^2\,|\, F$, as desired. The second part of the proposition follows by definition as $\frac F{(x_i-x_j)^2}$ is smooth.
\end{proof}

\begin{Proposition}\label{dlograt}
Let $h_{i,j}=\frac{f_i}{f_j}$. If $\operatorname{dlog}(h_{i,j})$ is not a rational function, then $Q_m(n)\subset(x_i-x_j)^{2m+1}\mathbb C[x_1,\dots,x_n]$. Here, $\operatorname{dlog}(f)=\frac{f'}f$ denotes the logarithmic derivative of a function $f$.
\end{Proposition}

\begin{proof}
The proposition is trivial for $m=0$. For $m>0$, note that $F\in Q_m(n)$ if and only if \begin{gather*}
f_i(x_i)f_j(x_j)F(x_i,x_j)-f_j(x_i)f_i(x_j)F(x_j,x_i)=(x_i-x_j)^{2m+1}g_a(x_i,x_j),\\
h_{i,j}(x_i)F(x_i,x_j)-h_{i,j}(x_j)F(x_j,x_i)=(x_i-x_j)^{2m+1}g_b(x_i,x_j).\end{gather*} Here, we treat the rest of the functions and variables as constants. Differentiating with respect to $x_i$, we have that \[h_{i,j}(x_i)(\operatorname{dlog}(h_{i,j})(x_i)F(x_i,x_j)+F_1(x_i,x_j))-h_{i,j}(x_j)F_2(x_j,x_i)=(x_i-x_j)^{2m}g_c(x_i,x_j),\]
where for a function $F(x,y)$ we define $F_1=\frac{\partial F}{\partial x}$ and $F_2=\frac{\partial F}{\partial y}$. Setting $x_i=x_j=x$, we have that \begin{gather*}
h_{i,j}(x)(\operatorname{dlog}(h_{i,j})(x)F(x,x)+F_1(x,x))-h_{i,j}(x)F_2(x,x)=0,\\
\operatorname{dlog}(h_{i,j})(x)F(x,x)=F_2(x,x)-F_1(x,x),
\end{gather*} which means that $F(x,x)=0$. Otherwise, we would have \[\operatorname{dlog}(h_{i,j})(x)=\frac{F_2(x,x)-F_1(x,x)}{F(x,x)},\] which is a contradiction as the right hand side is a rational function. Hence, $(x_i-x_j)\,|\, F$. Now, by Lemma~\ref{prop:squ} we have $(x_i-x_j)^2\,|\, F$ and $\frac F{(x_i-x_j)^2}\in Q_{m-1}(n)$, which implies the desired result by a straightforward induction.
\end{proof}

\subsection[Hilbert series for $n=2$]{Hilbert series for $\boldsymbol{n=2}$}\label{subsec:free}

Let $n=2$, $x=x_1$, and $y=x_2$. Note that scaling $f_1$ and $f_2$ by some smooth function does not affect $Q_m(2)$. Hence, we may multiply them both by $\frac1{f_2}$ and let $f=\frac{f_1}{f_2}$. For convenience, we use $Q_m(f)$ to denote the space of quasi-invariants. Throughout this section, we will let $\operatorname{dlog}(f(x))=\frac{p(x)}{q(x)}$ for relatively prime $p,q\in\mathbb C[x]$, as we have from Section~\ref{subsec:dlog} that either $Q_m=(x-y)^{2m}\mathbb C[x,y]$ or $\operatorname{dlog}(f(x))$ is a rational function. For convenience, we will also set $F_1=\frac{\partial F}{\partial x}$, $F_2=\frac{\partial F}{\partial y}$, and $F_{12}=\frac{\partial^2 F}{\partial x\partial y}$.

\begin{Lemma}\label{lem:ind}
If $F(x,y)\in Q_m(f)$, then
\begin{gather*} p(x)F_2(x,y)+q(x)F_{12}(x,y)\in Q_{m-1}\left(\frac fq\right), \\ -p(y)F_1(x,y)+q(y)F_{12}(x,y)\in Q_{m-1}(fq).
\end{gather*}
\end{Lemma}

\begin{proof}
We begin with our quasi-invariant condition, which in our case of $n=2$ is \[f(x)F(x,y)-f(y)F(y,x)=(x-y)^{2m+1}g_a(x,y).\]
Differentiating by $x$ and then $y$, we obtain \[f'(x)F_2(x,y)+f(x)F_{12}(x,y)-f'(y)F_2(y,x)-f(y)F_{12}(y,x)=(x-y)^{2m-1}g_b(x,y).\]
By the definition of $p$ and $q$, this is equivalent to \begin{gather*}
\frac{f(x)}{q(x)}\left(p(x)F_2(x,y)+q(x)F_{12}(x,y)\right)-\frac{f(y)}{q(y)}\left(p(y)F_2(y,x)+q(y)F_{12}(y,x)\right)\\
\qquad {} =(x-y)^{2m-1}g_b(x,y),
\end{gather*} which is exactly the quasi-invariant condition that is desired.

Dividing our quasi-invariant condition by $f(x)f(y)$ gives \[\frac1{f(x)}F(y,x)-\frac1{f(y)}F(x,y)=(x-y)^{2m+1}g_c(x,y).\]
(Note that $g_c$ is a function smooth on $D$, because $\frac{1}{f(x)f(y)}$ is smooth on $D$ by assumption.)
Thus, $-p(x)F_1(y,x)+q(x)F_{12}(y,x)\in Q_{m-1}\big(\frac1{fq}\big)$ by the above. Expanding the quasi-invariant condition and multiplying by $f(x)f(y)q(x)q(y)$, we obtain the equivalent statement $-p(y)F_1(x,y)+q(y)F_{12}(x,y)\in Q_{m-1}(fq)$, as desired.
\end{proof}

Now, we specialize to the case in which $f(x)=\prod_{i=1}^k(x-a_i)^{b_i}$ for arbitrary complex numbers~$a_i$,~$b_i$. Note that in this case $\operatorname{dlog}(f)=\sum_{i=1}^k\frac{b_i}{x-a_i}$.

\begin{Definition}
For a nonnegative integer $m$ and a complex number $z$, denote \[d_m(z)=\begin{cases}\min(m,|z|) & \text{if }z\in\mathbb{Z},\\m & \text{otherwise}, \end{cases}\]
and \[d_m(f)=\sum_{i=1}^k d_m(b_i),\] where $f(x)=\prod_{i=1}^k(x-a_i)^{b_i}$ for $a_1,\dots,a_k,b_1,\dots,b_k\in\mathbb C$ with $a_1,\dots,a_k$ pairwise distinct.
\end{Definition}

\begin{Lemma}\label{lem:div}
 We have that \[\prod_{i=1}^k(x-a_i)^{d_m(b_i)}\,|\, F(x,x)\] for any $F\in Q_m(f)$.
\end{Lemma}

\begin{proof}
We proceed using induction, with the base case of $m=0$ clearly true. Recall that $\operatorname{dlog}(f(x))=\frac{p(x)}{q(x)}$. For $f(x)=\prod_{i=1}^k(x-a_i)^{b_i}$, we have that $p(x)=\sum_{i=1}^kb_i\prod_{j\ne i}(x-a_i)$ and $q(x)=\prod_{i=1}^k(x-a_i)$. It suffices to prove this divisibility for each $(x-a_i)^{d_m(b_i)}$. By the inductive hypothesis and Lemma~\ref{lem:ind}, we have that $(x-a_i)^{d_{m-1}(b_i-1)}\,|\, p(x)F_2(x,x)+q(x)F_{12}(x,x)$ and $(x-a_i)^{d_{m-1}(b_i+1)}\,|\,-p(x)F_1(x,x)+q(x)F_{12}(x,x)$. It is easy to see that $d_m(b_i)-1\le d_{m-1}(b_i),d_{m-1}(b_i-1),d_{m-1}(b_i+1)$.

Thus, $(x-a_i)^{d_m(b_i)-1}\,|\, (x-a_i)^{d_{m-1}(b_i)}\,|\, F(x,x)$ as $F$ is also in $Q_{m-1}(f)$. From the other two divisibilities we also obtain $(x-a_i)^{d_m(b_i)-1}\,|\, p(x)F_2(x,x)+q(x)F_{12}(x,x)$, and $-p(x)F_1(x,x)+q(x)F_{12}(x,x)$, so $(x-a_i)^{d_m(b_i)-1}\,|\, p(x)(F_1(x,x)+F_2(x,x))=p(x)\frac{{\rm d}F(x,x)}{{\rm d}x}(x,x)$. As $p$ and $q$ are relatively prime, we must have $(x-a_i)^{d_m(b_i)-1}\,|\,\frac{{\rm d}F(x,x)}{{\rm d}x}(x,x)$ which together with $(x-a_i)^{d_m(b_i)-1}\,|\, F(x,x)$ implies $(x-a_i)^{d_m(b_i)}\,|\, F(x,x)$, as desired.
\end{proof}

In fact, this lemma is sharp in the sense that there exists $F$ such that the divisibility becomes equality. To prove that, we utilize the following lemma:

\begin{Lemma}\label{prop:prod}
If $F\in Q_m(f)$ and $G\in Q_m(g)$, then $FG\in Q_m(fg)$.
\end{Lemma}

\begin{proof}
We have that
\begin{gather*}
\frac{f(x)g(x)F(x,y)G(x,y)-f(y)g(y)F(y,x)G(y,x)}{(x-y)^{2m+1}}\\
\qquad{} =g(x)G(x,y) \frac{f(x)F(x,y)-f(y)F(y,x)}{(x-y)^{2m+1}} +f(y)F(y,x) \frac{g(x)G(x,y)-g(y)G(y,x)}{(x-y)^{2m+1}}
\end{gather*} is smooth.
\end{proof}

\begin{Lemma}\label{lem:gen}
There exists $P_m\in Q_m(f)$ with \[P_m(x,x)=\prod_{i=1}^k(x-a_i)^{d_m(b_i)}.\]
\end{Lemma}

\begin{proof}
Note that by Lemma~\ref{prop:prod} it suffices to show this when $k=1$ as for $k>1$ we can take the product of all such $P_m$ in $Q_m\big((x-a_1)^{d_m(b_1)}\big),\dots,Q_m\big((x-a_k)^{d_m(b_k)}\big)$. Shifting, we may also assume that $a_1=0$. Now, note that if $z=b_1$ is an integer less than $m$, we can simply take $P_m=y^{z}$. Otherwise, we claim that we can take \[P_m(x,y)=\frac{\sum\limits_{i=0}^m\binom{m-z}i\binom{m+z}{m-i}x^iy^{m-i}}{\binom{2m}m}.\]
Indeed, note that we have that $P_m(x,x)=x^m$ by Vandermonde's identity, so it suffices to show that $P_m$, or equivalently the numerator of $P_m$, is in $Q_m$. We proceed using induction, with the base case of $m=0$ obvious. For the inductive step, note that we wish to show that \[H(x,y):=\sum_{i=0}^m\binom{m-z}i\binom{m+z}{m-i}\big(x^{i+z}y^{m-i}-x^{m-i}y^{i+z}\big)\] vanishes at $x=y$ to order $2m+1$.

It is easy to see that $H(x,y)$ vanishes at $x=y$. Let us first show that $H(x,y)$ vanishes at $x=y$ to order $2$.
Differentiating with respect to $x$ and setting $x=y$, we would like to show that \[\sum_{i=0}^m\binom{m-z}i\binom{m+z}{m-i}(2i-m+z)=0.\]
As we have \[\sum_{i=0}^m\binom{m-z}i\binom{m+z}{m-i}i=\sum_{i=1}^m(m-z)\binom{m-z-1}{i-1}\binom{m+z}{m-i}=(m-z)\binom{2m-1}{m-1}\] and \begin{gather*}
\sum_{i=0}^m\binom{m-z}i\binom{m+z}{m-i}(i-m)=-\sum_{i=0}^{m-1}(m+z)\binom{m-z}i\binom{m+z-1}{m-i-1}\\
\hphantom{\sum_{i=0}^m\binom{m-z}i\binom{m+z}{m-i}(i-m)}{} =-(m+z)\binom{2m-1}{m-1}
\end{gather*} by Vandermonde's identity, the expression reduces to \[-2z\binom{2m-1}{m-1}+\sum_{i=0}^mz\binom{m-z}i\binom{m+z}{m-i}=-2z\binom{2m-1}{m-1}+z\binom{2m}m=0\] as desired.

Secondly, let us show that $\frac{\partial^2 H}{\partial y \partial x}$ vanishes at $x = y$ to order $2m-1$. Differentiating by both~$x$ and~$y$, it suffices to show that \[\sum_{i=0}^m\binom{m-z}i\binom{m+z}{m-i}(i+z)(m-i)\big(x^{i+z-1}y^{m-i-1}-x^{m-i-1}y^{i+z-1}\big)\] vanishes at $x=y$ to order $2m-1$. But note that this expression is
\begin{gather*}
\sum_{i=0}^{m-1}\binom{m-z}i\binom{m+z-1}{m-1-i}(i+z)(m+z)\big(x^{i+z-1}y^{m-i-1}-x^{m-i-1}y^{i+z-1}\big)\\
\qquad{} =(m+z)\sum_{i=0}^{m-1}\binom{m-z}i\binom{m+z-2}{m-1-i}(m+z-1)\big(x^{i+z-1}y^{m-i-1}-x^{m-i-1}y^{i+z-1}\big)\\
\qquad{} =(m+z)(m+z-1)\sum_{i=0}^{m-1}\binom{m-z}i\binom{m+z-2}{m-1-i}\big(x^{i+z-1}y^{m-i-1}-x^{m-i-1}y^{i+z-1}\big),
\end{gather*}
which vanishes at $x=y$ to order $2m-1$ by the inductive hypothesis on $Q_{m-1}(x^{z-1})$, as desired.

Thus we have seen that $H(x,y)$ vanishes at $x=y$ to order $2$ and \[\frac{\partial^2 H}{\partial y \partial x}=(x-y)^{2m-1}K(x,y)\] for certain polynomial $K(x,y)$. Since $H(x,y)$ vanishes at $x=y$ to order $2$, the solution of the above equation is unique. We can use the integration by parts to get an expression of the polynomial $H$ in terms of $(x-y)^k$, $k\ge 2m+1$, and the derivatives of $K(x,y)$ (the constant of integration is chosen to be zero). In this way one checks that $H(x,y)$ vanishes at $x=y$ to order $2m+1$.
\end{proof}

\begin{Lemma}\label{lem:indu}Let $R$ denote the ring of symmetric polynomials in $x$ and $y$. Then, for all $m>0$ we have that \[Q_m=RP_m+(x-y)^2Q_{m-1}.\]
\end{Lemma}

\begin{proof}
Let $F(x,y)$ be an element of $Q_m$. By Lemma~\ref{lem:div}, $P_m(x,x)\,|\, F(x,x)$, so there exists a~polynomial $g\in\mathbb C[x]$ with $P_m(x,x)g(x)=F(x,x)$. Now, consider the polynomial \[F'(x,y)=F(x,y)-P_m(x,y)g\left(\frac{x+y}2\right),\] which is in $Q_m$ as $F,P_m\in Q_m$ and $g\left(\frac{x+y}2\right)\in R$. But now note that $F'(x,x)=F(x,x)-P_m(x,x)g(x)=0$, so by Lemma~\ref{prop:squ}, $F'\in (x-y)^2Q_{m-1}$, which immediately implies the desi\-red.
\end{proof}

\begin{Corollary}
We have that \[Q_m=RP_m+R(x-y)^2P_{m-1}+\cdots+R(x-y)^{2m-2}P_1+(x-y)^{2m}Q_0\] for all $m$.
\end{Corollary}

Now, we are finally ready to prove our main result of this section. Recall that $d_m(f)=\sum_{i=1}^kd_m(b_i)$ where $f(x)=\prod_{i=1}^k(x-a_i)^{b_i}$ and $d_m(z)=\min(m,|z|)$ if $z\in\mathbb Z$ and $d_m(z)=m$ otherwise.

\begin{Theorem}\label{thm:main}
The Hilbert series for $Q_m(f)$ is \[\dfrac{t^{2m}+t^{2m+1}+\sum_{i=1}^m t^{2(m-i)+d_i(f)}-\sum_{i=1}^m t^{2(m-i)+d_i(f)+2}}{(1-t)(1-t^2)}.\]
\end{Theorem}

\begin{proof}
 Note that $Q_0=\mathbb C[x,y]$, which is generated by $P_0=1$ and $x-y$ (as an $R$ modules). By the corollary of Lemma~\ref{lem:indu}, $Q_m$ is generated by \[P_m, \ (x-y)^2P_{m-1}, \ \dots, \ (x-y)^{2m-2}P_1, \ (x-y)^{2m}, \ (x-y)^{2m+1}.\]
Let $g_{m,i}=(x-y)^{2(m-i)}P_i$ and $g_m=(x-y)^{2m+1}$. We claim that $Q_m$ is generated by $g_m, g_{m,0}, \dots, g_{m,m}$ and $m$ independent relations of the form
\begin{gather*}
(x-y)^2g_{m,m}=r_{m,m-1}g_{m,m-1}+\cdots+r_{m,0}g_{m,0}+r_mg_m,\\
 (x-y)^2g_{m,m-1}=r_{m-1,m-2}g_{m,m-2}+\cdots+r_{m-1,0}g_{m,0}+r_{m-1}g_m,\\
 \cdots\cdots\cdots\cdots\cdots\cdots\cdots\cdots\cdots\cdots\\
 (x-y)^2g_{m,1}=r_{1,0}g_{m,0}+r_1g_m
 \end{gather*}
for some $r_i,r_{i,j}\in R$. We proceed using induction, noting that $Q_0$ is generated by $1$ and $x-y$ with no relations. For the inductive step, first note that as $g_{m,m}=P_m\in Q_m\subset Q_{m-1}$, there exist $p, p_0, \dots, p_{m-1}\in R$ with $g_{m,m}=pg_{m-1}+p_0g_{m-1,0}+\dots+p_{m-1}g_{m-1,m-1}$. This yields a relation in the form of the first relation above by setting $r_m=p, r_{m,0}=p_0, \dots, r_{m,m-1}=p_{m-1}$. This equation is true as $g_m=(x-y)^2g_{m-1}, g_{m,0}=(x-y)^2g_{m-1,0}, \dots, g_{m,m-1}=(x-y)^2g_{m-1,m-1}$ by definition. Now, suppose that $q, q_0, \dots, q_m\in R$ such that $qg_m+q_0g_{m,0}+\dots+q_mg_{m,m}=0$. Then, as $(x-y)^2\,|\, g_m,g_{m,0},\dots,g_{m,m-1}$ and $(x-y)^2\nmid g_{m,m}$, we must have that $(x-y)^2\,|\, q_m$. Let $q_m=(x-y)^2q_m'$. Then, subtracting $q_m'$ times the first relation from $qg_m+q_0g_{m,0}+\dots+q_mg_{m,m}=0$, we obtain a relation of the form $q'g_m+q_0'g_{m,0}+\dots+q_{m-1}'g_{m,m-1}=0$ with $q', q_0', \dots, q_{m-1}'\in R$. Note that this relation is uniquely determined by the first generating relation we have so the first generating relation is independent of the rest of the relations. Furthermore, this relation is $(x-y)^2$ times a relation among the generators of $Q_{m-1}$. By the inductive hypothesis, such a relation is generated by $(x-y)^2$ times the $m-1$ independent generating relations of $Q_{m-1}$, which are by definition the last $m-1$ generating relations on our list. Hence, $Q_m$ is generated by those $m+2$ elements and $m$ independent relations among those elements, as desired.

For the Hilbert polynomial, note that the generators have degrees \[d_m(f),\ 2+d_{m-1}(f),\ \dots,\ 2m-2+d_1(f),\ 2m,\ 2m+1\] and that the independent relations have degrees \[2+d_m(f),\ 4+d_{m-1}(f),\ \dots,\ 2m+d_1(f),\] which gives the Hilbert polynomial and series exactly as described in the theorem.
\end{proof}

\section{Future prospects}\label{sec:fut}
It would be interesting to study Conjectures~\ref{conj:mn},~\ref{conj:mainc} and~\ref{conj:big}. As with our current results, we expect to make extensive use of computer programs to discover key properties of quasi-invariant polynomials and their Hilbert series. We expect that resolving Conjecture~\ref{conj:mn} will require studying the modular representation theory of~$S_n$.

A possible approach to Conjecture~\ref{conj:big} is to adapt the approach of the authors of \cite{braverman2016cyclotomic}, namely to construct a Cherednik-like algebra related to $f_1,\dots,f_n$ and the quasi-invariant polynomials. Along the way, one may also find the formula of the Hilbert series.

Finally, it would be interesting to study $q$-deformations of the spaces of twisted quasi-invariant polynomials. In~\cite{braverman2016cyclotomic}, Braverman, Etingof and Finkelberg study $q$-deformations of their special case and show that when $Q_m$ is free, its $q$-deformation is a flat deformation. They conjecture that it is a flat deformation in general even when $Q_m$ is not a free module. Here, $Q_{m,q}(f_1,\dots,f_n)$ is defined as the set of polynomials $F$ for which \[\frac{(1-s_{i,j})(f_1(x_1)\cdots f_n(x_n)F)}{\prod\limits_{k=-m}^m\big(x_i-q^kx_j\big)}\] is a smooth function for all $1\le i<j\le n$. It would be interesting to resolve this in the case that Braverman, Etingof, and Finkelberg consider, as well as the general case we have presented. We believe that $q$-analogues of some of our results hold. For example, the $q$-analogue of Proposition~\ref{dlograt} would be that $Q_{m,q}\subset\prod_{k=-m}^m\big(x_i-q^kx_j\big)$ if $\frac{h_{i,j}(qx)}{h_{i,j}(x)}$ is not rational.

\subsection*{Acknowledgements}
We would like to thank MIT PRIMES, specifically Pavel Etingof, for suggesting the project. We would like to thank Eric Rains for very useful discussions. We also would like to thank the referees for carefully reading our manuscript and for their valuable comments and suggestions which substantially help to improve the readability and quality of the paper.

\pdfbookmark[1]{References}{ref}
\LastPageEnding

\end{document}